\documentclass[a4paper]{amsart}  

\usepackage{amsmath}
\usepackage{amssymb}
\usepackage{amsthm}
\usepackage{mathtools}
\usepackage{esint}
\usepackage{xspace}
\usepackage{enumitem}
\usepackage{algorithm}
\usepackage{algpseudocode}

\setlist[enumerate]{leftmargin=5ex}

\theoremstyle{plain}
\newtheorem{theorem}{Theorem}

\newtheorem{lemma}[theorem]{Lemma}

\theoremstyle{remark}
\newtheorem{remark}[theorem]{Remark}
\newtheorem{example}[theorem]{Example}

\newcommand{\N}{\mathbb{N}}
\newcommand{\R}{\mathbb{R}}
\newcommand{\Rd}{\R^d}
\newcommand{\opnint}[2]{{]}#1,#2{[}}
\newcommand{\clsint}[2]{[{#1},{#2}]}

\DeclareMathOperator{\Conv}{conv}

\DeclareMathOperator{\Vertices}{vert}

\newcommand{\norm}[2]{\left\|#1\right\|_{#2}}

\newcommand{\anc}{\mathcal{A}} 
 \newcommand{\manc}{\hat{\anc}} 
 \newcommand{\mcmA}{\hat{p}} 
\newcommand{\besterr}{\sigma}
\newcommand{\err}{e} 
\newcommand{\Err}{E} 
\newcommand{\cell}{\Delta} 
 \newcommand{\mcell}{\cell^\star} 
 \newcommand{\tcell}{K} 
\newcommand{\ind}{\mu} 
 \newcommand{\mind}{\hat{\ind}} 
 \newcommand{\auxind}{\nu} 
 \newcommand{\mauxind}{\hat{\nu}} 
\newcommand{\mesh}{M}\newcommand{\imesh}{\mesh_\text{init}} 
\newcommand{\leaves}{\mathcal{L}} 
\newcommand{\maxpatchsize}{\Theta} 
\DeclareMathOperator{\redge}{re}
\newcommand{\sdpatch}{s} 
 \newcommand{\npatches}{S} 
 \newcommand{\apatches}{S^+} 
\newcommand{\parent}{p} 
\newcommand{\Poly}{\mathbb{P}} 
\newcommand{\RootOfTree}{\Omega} 
\newcommand{\tree}{T} 
 \newcommand{\mtree}{\hat{\tree}} 
 \newcommand{\mastertree}{\tree_{\RootOfTree}} 
 \newcommand{\ctrees}{\mathbb{T}^\text{C}} 
\begin{document}
%
%
%

%
%
%
%

\title[Near-best adaptive approximation on conforming meshes]%
{Near-best adaptive approximation \\ on conforming meshes}


\author{Peter Binev}
\address{Peter Binev, Department of Mathematics, University of South Carolina, Columbia, SC 29208, USA}
\email{binev@math.sc.edu}
\urladdr{people.math.sc.edu/binev/}
\author{Francesca Fierro}
\address{Francesca Fierro, Dipartimento di Matematica,
	Universit\`a degli Studi di Milano,
	Via~Saldini 50,
	20133 Milano, Italia
}
\email{francesca.fierro@unimi.it}
\urladdr{www.mat.unimi.it/users/fierro/}
\author{Andreas Veeser}
\address{Andreas Veeser, Dipartimento di Matematica,
	Universit\`a degli Studi di Milano,
	Via~Saldini 50,
	20133 Milano, Italia.
	Member of the INdAM research group GNCS
\\
\phantom{123}\emph{ORCID:} \texttt{0000-0002-2152-2911}
\\
\phantom{123}\emph{ResearcherID:} \texttt{J-2062-2018}
}
\email{andreas.veeser@unimi.it}
\urladdr{www.mat.unimi.it/users/veeser/}

\begin{abstract}
We devise a generalization of tree approximation that generates conforming meshes, i.e., meshes with a particular structure like edge-to-edge triangulations. A key feature of this generalization is that the choices of the cells to be subdivided are affected by that particular structure. As main result, we prove near best approximation with respect to conforming meshes, independent of constants like the completion constant for newest-vertex bisection. Numerical experiments complement the theoretical results and indicate better approximation properties than previous approaches. 
\end{abstract}

\maketitle

\section{Introduction}
\label{S:Intro}
%
%
%
In approximation with adaptive meshes, one picks a cell of the current mesh in a target-function specific way and then subdivides it to generate the next mesh. This iterative process, which is accompanied by an incremental addition of degrees of freedom, is of great interest. In fact, it is underlying or even part of adaptive finite element methods, one of the most successful approaches to solve numerically partial differential equations. The mesh cells then typically have to meet in a certain manner to ensure some smoothness of the approximants or to allow for some algorithmic convenience.

In a prominent setup, the cells are triangles, subdivision is realized by newest vertex bisection, and the triangulations have to be edge-to-edge. Newest vertex bisection subdivides a given triangle with assigned newest vertex by drawing a line from the newest vertex to the midpoint of the opposite edge, which becomes the newest vertex of the two new triangles.
Note that the subdivision history and so triangulations correspond to binary trees. Edge-to-edge means that the intersection of two different triangles is either empty, a vertex, or a common edge. This property is also called ``conforming'', while here we use this less specific term for indicating a generic requirement.

In order to motivate and illustrate our results in the introduction, let us restrict ourselves to this setup described in the preceding paragraph. We write $\mathbb{T}$ and $\ctrees$ for, respectively, the set of all trees with general triangulations and its subset of trees with edge-to-edge triangulations. Moreover, if $\tree \in \mathbb{T}$, then $\Err(\tree)$ denotes a suitable error associated with the triangulation of $\tree$.

A particularly convenient way of adaptive approximation is the tree algorithm in Binev \cite[\S2]{binev_tree_2018}. This algorithm however outputs only trees in $\mathbb{T}$ with triangulations that are not necessarily edge-to-edge. We may therefore add a completion step, i.e., we pass to the smallest tree in $\ctrees$ with an edge-to-edge triangulation that contains the output of the tree algorithm. The modified output tree $\tree$ is then near best in the sense that
\begin{equation}
\label{nb-with-completion}
 \Err(\tree)
 \leq
 4 \inf \left\{ \Err(\tree') \mid
  \tree' \in \mathbb{T},\; \#\tree' \leq \tfrac{3}{4} C_\text{cmpl}^{-1}\#T
  \right\},
\end{equation}
where the constant $C_\text{cmpl}$ stems from the completion step and, in general, can be quite large. Its presence is therefore not desirable, although justified by the fact that the competing trees on the right-hand side are not necessarily edge-to-edge.

In order to avoid the presence of the completion constant $C_\text{cmpl}$, we thus shall aim for an inequality like \eqref{nb-with-completion} where $\mathbb{T}$ is replaced by $\ctrees$. We achieve this goal by devising a new tree algorithm with the following features:
\begin{itemize}
\item The generated triangulations are edge-to-edge at any stage of the algorithm.
\item The choices of the cells to be subdivided depend on the edge-to-edge requirement by separating the actual subdivision from the request for it.
\item Without the edge-to-edge requirement, the new tree algorithm reduces to the one in Binev \cite[\S2]{binev_tree_2018}.
\end{itemize}
The counterpart of \eqref{nb-with-completion} for an output tree $\tree$ of the new algorithm reads
\begin{equation*}
\Err(\tree)
\leq
4 \inf \left\{ \Err(\tree') \mid
\tree' \in \mathbb{T},\; \#\tree' \leq \tfrac{1}{2} \#\tree
\right\}.
\end{equation*}
We complement this near-best result by numerical examples. They indicate that the approximation properties of new tree algorithm are at least as good as the ones of the algorithm with completion. In certain situations, e.g., for triangulations of small cardinality or partially trivial target functions, they are even better. Furthermore, the numerical results suggest that a more relevant improvement will take place for setups with larger completion constant $C_\text{cmpl}$.

The rest of this manuscript is organized as follows. Section \ref{S:TreeApprox-for-ConfMeshes} describes the general setting of conforming meshes that is covered by our results, devises the new tree algorithm, and proves that its outputs are near best. The numerical experiments are presented in \S\ref{S:Tests}, along with a motivation of the considered setting.

\section{Tree approximation for conforming meshes}
\label{S:TreeApprox-for-ConfMeshes}
\subsection{A tree setup for conforming meshes}
\label{S:setup}
In this subsection, we specify the type of adaptivity to be considered. It is based upon \emph{trees}, i.e.\ connected acyclic undirected graphs.

Let $\Omega$ be a domain in $\Rd$, $d\in\N$. We suppose that we are given an infinite tree $\mastertree$ with \emph{root} $\RootOfTree$ such that, for every node or cell $\cell \in \mastertree$, its \emph{children} form a finite partition of $\cell$. More precisely, the children $\cell_1,\dots,\cell_J$, where $J \geq 2$ can depend on $\cell$, are open subsets of $\Rd$ and such that $\overline{\cell} = \bigcup_{j=1}^J \overline{\cell}_j$ and, for $i \neq j$, the intersection $\cell_i \cap \cell_j$ is empty. 
The replacement of $\cell$ by its children is a \emph{subdivision}. We also say that $\cell$ is the \emph{parent} of each $\cell_j$ and set $\parent(\cell_j) := \cell$ as well as $\parent(\RootOfTree) := \emptyset$. 

Let $\tree$ be a subtree of $\mastertree$. The subtree $\tree$ is called \emph{finite} whenever the number of its nodes is finite: $\#\tree < \infty$. A node that has no children in $\tree$ is a \emph{leaf} of $\tree$ and we write $\leaves(\tree)$ for the set of all its leaves. If every non-leaf node of $\tree$  
has the same number of children in $\tree$ and $\mastertree$,
then the tree $\tree$ is called \emph{full}, otherwise \emph{general}. Obviously, if a subtree $\tree \subseteq \mastertree$ is finite, full, and rooted at $\RootOfTree$, its leaves $\leaves(\tree)$ form a finite partition of $\RootOfTree$. Another useful type of subtree is the \emph{tree of all descendants} $\tree_\cell$ of a node $\cell \in \mastertree$; this tree is rooted at $\cell$, full, and,  not having any leaves, is infinite. If $\cell' \in \tree_\cell$, we all say that $\cell$ is an \emph{ancestor} of $\cell'$ and write $\cell \in \mathcal{A}(\cell')$. 

The partitions induced by full subtrees are quite general. In applications, it is often desirable to allow only for partitions with some additional structure, to which we will refer as \emph{conforming meshes}. Before giving examples, we shall introduce related notation and assumptions.

A subtree $\tree\subset\mastertree$ that is full and rooted at $\RootOfTree$ is a \emph{conforming tree} whenever its leaves form a conforming mesh. We denote by $\ctrees$ the set of all subtrees of $\mastertree$ that are rooted at $\RootOfTree$, full, and conforming.

Let $\sdpatch$ be a set of cells appearing in some conforming mesh. The subdivisions of $\sdpatch$ are \emph{isolated} whenever, for any conforming mesh containing $\sdpatch$, the replacement of all cells in $\sdpatch$ by their children leads again to a conforming mesh. If $\sdpatch$ is minimal (in the sense that there is no proper subset with isolated subdivisions), we call it a \emph{subdivision patch}. We assume 
\begin{equation}
\label{sdpatch-for-cell}
\begin{aligned}
 &\text{for any cell $\cell \in \mastertree$,}
\\
 &\quad \text{there is a unique subdivision patch $\sdpatch(\cell)$ with $\cell \in \sdpatch(\cell)$}.
\end{aligned}
\end{equation}
The uniqueness entails that, if $\sdpatch$ is a subdivision patch, then
\begin{equation}
\label{patch-and-members}
 \cell,\cell' \in \sdpatch
 \iff
 \sdpatch(\cell) = s = \sdpatch(\cell').
\end{equation}

If $\cell \in \leaves(\tree)$ is a leaf of a conforming tree $\tree\in\ctrees$, its subdivision patch might not satisfy $\sdpatch(\cell) \subseteq \leaves(\tree)$. If not, $\cell$ cannot be immediately subdivided in $\leaves(\tree)$; its subdivision has to be prepared by other ones. A subdivision patch $\sdpatch$ is \emph{necessary} for subdividing the cell $\cell$ whenever
\begin{equation}
\label{locating-subdivision-patches}
 \forall \tree \in \ctrees
\quad
 \cell \in \tree \setminus \leaves(\tree)
 \implies
 \sdpatch \subseteq \tree \setminus \leaves(\tree).
\end{equation} 
In particular, $\sdpatch(\cell)$ is necessary for the subdivision of $\cell$ by its minimality. We assume that there always exists at least one necessary subdivision patch:
\begin{equation}
\label{existence-necessary-subdivison-patches}
\begin{aligned}
 &\text{for any leaf $\cell \in \leaves(\tree)$ of a conforming tree $\tree \in \ctrees$,}
\\
 &\quad\text{there is a subdivision patch $s \subseteq \leaves(\tree)$ that is necessary for subdividing $\cell$}.
\end{aligned}
\end{equation}
This assumption ensures that any conforming mesh can be constructed by successively subdividing necessary patches. 

In our results below, we measure the size of a conforming tree $\tree \in \ctrees$ by
\begin{equation}
\label{complexity-measure}
|\tree|
:=
\# \{ s \subseteq \tree \setminus \leaves(\tree) \mid
s \text{ is a subdivision patch} \}
\end{equation}
and employ the maximal cardinality
\begin{equation}
\label{maxpatchsize}
\maxpatchsize
:=
\max_{\cell \in \mastertree} \#\sdpatch(\cell) \geq 1
\end{equation}
of a subdivision patch.

Let us now consider two 
examples, verifying the two assumptions \eqref{sdpatch-for-cell} and \eqref{existence-necessary-subdivison-patches}. For the sake of simplicity, we shall suppose here that there is only one type of cells in each example and that the domain $\Omega$ is an instance thereof. As long as the master tree is well-defined, examples involving various cell types can be handled quite similarly. More general domains and initial meshes with several cells are covered by Remark \ref{R:arb-init-mesh} below. The first example will serve us as a reference, while the second one constitutes our running example. 

\begin{example}[Nonconforming abstract bisection]
\label{E:nonconf-abstract-bisection}
Let $\mastertree$ be a binary tree and take $\ctrees$ to be the set of all subtrees that are full and rooted at $\RootOfTree$, i.e., there is no requirement of an additional structure for the meshes. This is the setup considered in Binev \cite[\S2]{binev_tree_2018}. Obviously, the subdivision patches are $\sdpatch(\cell) = \{\cell\}$, whence
\begin{subequations}
\label{setting-nonconf-abstract-bisection}
\begin{equation}
 \maxpatchsize=1
\end{equation}
and \eqref{sdpatch-for-cell} and \eqref{existence-necessary-subdivison-patches} are verified. Furthermore, for any $\tree \in \ctrees$, we have
\begin{equation}
 |\tree| = \#\big( \tree \setminus \leaves(\tree) \big) = \#\leaves(\tree) - 1,
\end{equation}
\end{subequations}
where the second inequalities follows from the fact that $\tree$ is a full binary tree.
\end{example}

\begin{example}[Newest vertex bisection for edge-to-edge triangulations]
\label{E:nvb-triangulations}
%
Given a triangle $\Omega\subset\R^2$, let $\mastertree$ be the infinite binary tree created by newest vertex bisection of triangles; cf., e.g., \cite{mitchell_30_2016}. In order to recall this subdivision method, we associate to a triangle $\cell$ an ordering of its vertices (or: extreme points) $\Vertices\cell := \{a_0,a_1,a_2\}$. The ordering $(a_0,a_1,a_2)$, which has to be assigned for the root triangle, is such that $a_1$ corresponds to the so-called newest vertex and $\redge(\cell) := \Conv\{a_0,a_2\}$ is the refinement edge. The two children of $\cell$ are then given by $(a_0,a,a_1)$ and $(a_2,a,a_1)$, where the midpoint $a = \tfrac{1}{2}(a_0+a_2)$ of the refinement edge $\redge(\cell)$ is the respective newest vertex.

Let $\tree \subseteq \mastertree$ be a tree that is full and rooted at $\RootOfTree$. We say that the mesh induced by its leaves or $\tree$ itself is \emph{edge-to-edge} whenever
\begin{equation*}
 \forall \cell_1,\cell_2 \in \leaves(\tree)
\quad
 \Vertices(\cell_1 \cap \cell_2) = \Vertices \cell_1 \cap \Vertices \cell_2
\end{equation*}
and take
\begin{equation*}
 \ctrees := \{ \tree \subseteq \mastertree \mid \tree \text{ is full, rooted at $\RootOfTree$, and edge-to-edge} \}.
\end{equation*}
Edge-to-edge meshes are convenient, e.g., in the construction of finite element bases that have a certain smoothness.

In order to verify assumptions \eqref{sdpatch-for-cell} and \eqref{existence-necessary-subdivison-patches}, let us first determine the possible configurations of a subdivision patch $\sdpatch$ in the given setup. Let $\cell \in \sdpatch$ be a cell of the subdivision patch $\sdpatch$. If its refinement edge $\redge(\cell) \subset \partial\Omega$ belongs to the boundary, then $\sdpatch = \{ \cell \}$, else there is another cell $\cell' \in \mastertree$ such that $\redge(\cell) = \redge(\cell')$ and $s = \{ \cell, \cell' \}$; see Binev, Dahmen, and DeVore \cite[Lemma~2.2~(ii)]{binev_adaptive_2004}.
Hence, \eqref{sdpatch-for-cell} is verified and we have $\maxpatchsize = 2$. Assumption \eqref{existence-necessary-subdivison-patches} follows from \cite[Lemma~2.2~(iii)]{binev_adaptive_2004}.
\end{example}

%

We next introduce our error notion. To this end, we associate a \emph{local error} $\err(\cell)\geq0$ with each cell $\cell \in \mastertree$ and, for any conforming tree $\tree \in \ctrees$, we define its \emph{global error} in an additive way by 
\begin{equation*}
 \Err(\tree)
 :=
 \sum_{\cell \in \leaves(\tree)} \err(\cell).
\end{equation*}
Examples for error functionals $\err$ will be presented in \S\ref{S:Tests}.

Given a budget of subdividing $n$ patches, the \emph{best (global) error} is
\begin{equation*}
 \besterr_n
 :=
 \min_{\tree \in \ctrees_n} \Err(\tree)
\end{equation*}
with $\ctrees_n := \{ \tree \in \ctrees \mid |\tree| \leq n \}$. We say that the trees constructed by some algorithm are \emph{near best} with respect $|\cdot|$ whenever there exist constants $C_1, C_2>0$ such that, for every output tree $\tree$,  
\begin{equation}
\label{def-near-best}
 \Err(\tree)
 \leq
 C_1 \besterr_{C_2|\tree|}.
\end{equation}

In order to construct near-best trees without investigating the global error functionals of all competing trees, the following assumption\ 
on the local errors will be instrumental. The local errors are \emph{subadditive} if 
\begin{equation}
\label{subadditivity}
 \sum_{j=1}^J \err(\cell_j)
 \leq
 \err(\cell)
\end{equation}
whenever $\cell_1, \dots,\cell_J$ are the children of a cell $\cell$. 

\subsection{Tree algorithm for conforming meshes}
\label{S:TreeAlgo}
%
%

A greedy algorithm on the local errors does not construct near-best meshes. This is illustrated by Example~4 in Veeser~\cite{veeser_adaptive_2018}, which considers bisection of intervals and local errors arising from a scalar product. Roughly speaking, due to a suitable orthogonality of the target function, an arbitrary number of subdivisions of the cells with the maximal local error yields no error reduction at all, while a relevant reduction could be obtained by subdividing other cells.

In order to remedy, Binev and DeVore \cite{binev_fast_2004} propose applying a greedy algorithm to ``modified local errors''. These quantities depend not only on the local error of the cell at hand but also on the local errors of its ancestors. We shall refer to them as \emph{marking indicators} and denote by $\ind(\cell)$ the marking indicator associated with the cell $\cell$. In the context of abstract nonconforming bisection (cf.\ Example \ref{E:nonconf-abstract-bisection}), a particularly convenient marking indicator is conceived in Binev~\cite{binev_tree_2018}. If $\err(\cell)>0$, it records the subdivision generating the cell $\cell$ in the following manner:
\begin{equation}
\label{recording-nonconforming-subdivisions}
 \frac{\err(\cell)}{\ind(\cell)}
 =
 1 + \frac{\err(\cell)}{\ind\big( \parent(\cell) \big)},
\end{equation}
i.e., the ratio of local error to marking indicator is increased by $1$ with respect to the value $\ind\big(\parent(\cell)\big)$ of the parent of $\cell$. This enlargement of the ratio constitutes a penalization for $\ind(\cell)$ being large and, in particular, enforces some reduction of the marking indicator when passing form parent to child. In fact, an equivalent form of the identity \eqref{recording-nonconforming-subdivisions} is 
\begin{equation}
\label{nonconforming-recursion}
 \ind(\cell)
 =
 \left(
  \frac{1}{\err(\cell)} + \frac{1}{\ind\big( \parent(\cell) \big)}
 \right)^{-1}.
\end{equation}
This recursion reveals that $\mu$ has a double meaning: on the one hand, $\ind(\cell)$ is the marking indicator replacing the local error $\err(\cell)$ and, on the other hand, $\ind(\parent(\cell))^{-1}$ can be viewed as the penalization for $\cell$ arising from previous subdivisions.

Before proposing a generalization of recursion \eqref{nonconforming-recursion}, we note two differences of conforming meshes from nonconforming ones. First, subdivision is organized in cell patches instead of single cells. Second, a cell $\mcell$ that is \emph{marked}, i.e.\ should be subdivided, might not be ready for subdivision since its patch $\sdpatch(\mcell)$ might not be contained in the current mesh. We can take these two aspects into account by recording subdivisions of patches instead of cells and by penalizing the request of patch subdivisions and not their realization. More precisely, we additionally introduce the \emph{``penalization indicator''} $\auxind(\cell)$ of a cell $\cell$. Then, if the cells of a patch have just been subdivided at the request of the marked cell $\mcell$, we assign new indicators by distinguishing the following three cases:
\begin{itemize}
\item If a cell $\cell$ is generated by the subdivision of the marked cell $\mcell$, we assign
\begin{equation*}
 \auxind(\cell) \gets \ind\big( \parent(\cell) \big)^{-1}
\quad\text{and}\quad
 \ind(\cell)
 \gets
 \left(
  \frac{1}{\err(\cell)} + \auxind(\cell)
 \right)^{-1}.
\end{equation*}
This corresponds exactly to \eqref{nonconforming-recursion}.
\item If a cell $\cell$ is generated, but not by subdividing the marked cell $\mcell$, we assign
\begin{equation*}
\label{conforming-recursion;for-conformity}
 \auxind(\cell) \gets \auxind\big( \parent(\cell) \big)
\quad\text{and}\quad
 \ind(\cell)
 \gets
 \left(
  \frac{1}{\err(\cell)} + \auxind(\cell)
 \right)^{-1}.
\end{equation*}
This does not modify the penalization and so does not record the subdivision of $\parent(\cell)$ for its descendants. However, the marking indicator of $\cell$ is different from the one of its parent whenever this hold for their local errors.
\item If the marked cell $\mcell$ is not subdivided, then we \emph{update} its marking and penalization indicators as follows:
\begin{equation*}
\label{conforming-recursion;elm-with-max-mind}
 \auxind(\mcell) \gets \ind( \mcell)^{-1}
\quad\text{and}\quad
 \ind(\mcell)
 \gets
 \left(
  \frac{1}{\err(\mcell)} + \auxind(\mcell)
 \right)^{-1}.
\end{equation*}
This corresponds to \eqref{nonconforming-recursion}, where $\mcell$ is its own parent.
\end{itemize}
In this way only the penalization indicators of the marked cell $\mcell$ or its descendants record one patch subdivision, irrespective of the number of subdivided cells and whether $\mcell$ was actually subdivided. 

The recursive assignments for the penalization indicator can be initialized by
$
 \auxind(\Omega) \gets 0.
$ 
If $\err(\cell) = 0$, we still use the above assignments, with the conventions $1/0 = \infty$, $r+\infty=\infty$ whenever $r\in\R\cup\{\infty\}$, and $1/\infty = 0$.

The preceding discussion leads to the pseudocode of Algorithm~\ref{A:conf-tree-algo}.
\begin{algorithm}
	\caption{Tree algorithm for conforming meshes}	
	\label{A:conf-tree-algo}
	\begin{algorithmic}[1]
		\Statex
		\Require master tree $\mastertree$ with local error functional $\err$
		\Statex
		\Statex \Comment{initialization}
		
		\State $\tree_0 \gets \{\Omega\}$, $\auxind(\Omega) \gets 0$, $\ind(\Omega) \gets \err(\Omega)$, $n \gets 0$
		\Statex
		\While {$t_n := \max_{\cell\in\leaves(\tree_n)} \ind(\cell)>0$}
		\label{max>0}
		\Comment{or some alternative stopping test}
		\State mark some $\mcell_n \in \leaves(\tree_n)$ with $\ind(\mcell_n) = t_n$
		\State pick a subdivision patch $\sdpatch_n \subseteq \leaves(\tree_n)$ that is necessary for subdividing $\mcell_n$
		\State define $\tree_{n+1}$ by adding all children $C_n$ of the cells in $\sdpatch_n$ to $\tree_n$
		\Statex
		\Statex \Comment{penalize only the marking indicators of the descendants of $\mcell_n$}
		\State $\auxind(\mcell_n) \gets
		 \displaystyle
		 \frac{1}{\err(\mcell_n)} + \auxind(\mcell_n)$
		\label{dynamic;1}
		\If {$\mcell_n \not\in \sdpatch_n$}
		\Comment{$\mcell_n$ remains a leaf}
		\State $\ind(\mcell_n) \gets \left( \displaystyle
		\frac{1}{\err(\mcell_n)} + \auxind(\mcell_n)
		\right)^{-1}$
		\EndIf
		\ForAll {$\cell \in C_n$}
		\State $\auxind(\cell) \gets \auxind\big(\parent(\cell)\big)$
		\State $\ind(\cell) \gets \left( \displaystyle
		\frac{1}{\err(\cell)} + \auxind(\cell)
		\right)^{-1}$
		\EndFor
		\State $n \gets n+1$
		\EndWhile
	\end{algorithmic}
\end{algorithm}
Three comments about it are in order. The last two are immediate consequences of the derivation of the marking and penalization indicators.
\begin{itemize}
\item Since  in each iteration the children of the cells in exactly one subdivision patch are subdivided, iteration counter and complexity measure coincide:
\begin{equation}
\label{iteration=complexity}
 \text{for all $n$, we have $|\tree_n| = n$.}
\end{equation} 
\item Algorithm~\ref{A:conf-tree-algo} specializes to the (unrestricted) tree algorithm in Binev~\cite[\S2]{binev_tree_2018} whenever $\sdpatch_n = \{ \mcell_n \}$ for all $n$, i.e.\ in particular for the setting of abstract nonconforming bisection in Example~\ref{E:nonconf-abstract-bisection}.
\item A marking indicator $\ind(\cell)$ does not depend only on the cell $\cell$ but also on the iteration number $n$. This is due to the above third case or 
line \ref{dynamic;1}, and it is an important difference from Binev and DeVore \cite{binev_fast_2004} and Binev~\cite[\S2]{binev_tree_2018}, where the marking indicator depends solely on $\cell$. 
\end{itemize}

In order to track the dependence on the iteration number $n$ carefully, let us incorporate it in our notation. If $\cell \in \leaves(\tree_n)$ is a mesh cell at iteration $n$, we write $\ind(\cell,n)$ and $\auxind(\cell,n)$ for the corresponding values of  $\ind(\cell)$ and $\auxind(\cell)$ available at iteration $n$ when executing line \ref{max>0}. Furthermore, if Algorithm~\ref{A:conf-tree-algo} terminates, then we let $N^*$ denote the last of iteration counter $n$ in line 2, else we set $N^* := \infty$. For the penalization, we then have the initialization
\begin{subequations}
\label{ind-with-ell}
\begin{equation}
\label{aux-ind;init}
 \auxind(\Omega,0) = 0,
\end{equation}
and, for all $n<N^*$ and $\cell\in\leaves(\tree_{n+1})$, the recursion
\begin{equation}
\label{aux-ind;recursion}
 \auxind(\cell,n+1)
 =
 \begin{cases}
  \displaystyle
   \frac{1}{\err(\mcell_n)}
   +
   \auxind(\mcell_n,n)
  &\text{if $\cell\not\in\leaves(\tree_{n})$ and $\parent(\cell) = \mcell_n$},
 \\[2ex]
  \auxind\big( \parent(\cell),n \big)
  &\text{if $\cell\not\in\leaves(\tree_{n})$ and $\parent(\cell) \neq \mcell_n$,}
 \\[2ex]
  \displaystyle
  \frac{1}{\err(\mcell_n)}
  +
  \auxind(\mcell_n,n)
  &\text{if $\cell\in\leaves(\tree_{n})$ and $\cell = \mcell_n$,}
 \\[2ex]
  \auxind(\cell,n)
  &\text{if $\cell\in\leaves(\tree_{n})$ and $\cell \neq \mcell_n$}.
 \end{cases}
\end{equation}
The marking indicators are simply given by
\begin{equation}
\label{marking-with-ell}
 \ind(\cell,n)
 =
 \left(
  \frac{1}{\err(\cell)}
  +
  \auxind(\cell,n)
 \right)^{-1}
\end{equation}
\end{subequations}
for all $n\leq N^*$ and $\cell\in\leaves(\tree_{n})$.

These relations and line \ref{max>0} of Algorithm \ref{A:conf-tree-algo} imply the following two properties of the indicators of a ``tagged'' cell $(\cell,n)$ with $\cell \in \leaves(\tree_n)$ and $n\leq N^*$: First, we have the strict inequality
\begin{equation}
\label{auxind>0}
 \auxind(\cell,n) > 0
\end{equation}
whenever $1\leq n \leq N^*$. Second, without restriction on $n$, we have
\begin{equation}
\label{err,ind=0}
 \ind(\cell,n) = 0 \iff \err(\cell) = 0. 
\end{equation}
This equivalence has an interesting consequence. If there exists a tree $\tree \in \ctrees$ with $\Err(\tree) = 0$, then Algorithm~\ref{A:conf-tree-algo} terminates with a subtree: $\tree_{N^*} \subseteq \tree$. This is a necessary condition that Algorithm~\ref{A:conf-tree-algo} generates near-best meshes with $C_2=1$.
\subsection{Near-best approximation for subadditive local errors}
\label{S:nba-subadditivity}
The goal of this section is to analyze the approximation properties of Algorithm~\ref{A:conf-tree-algo} under the assumption that the local errors are subadditive.


%
\begin{theorem}[Near-best approximation under subadditivity]
\label{T:nba-subadditivity}
Let $\tree_N$ be the tree generated by Algorithm~\ref{A:conf-tree-algo} at iteration $N$. If the local errors are subadditive, then the error of $\tree_N$ satisfies
\begin{equation*}
 \Err(\tree_N)
 \leq
 \min_{n=0}^N \left(
  \frac{N+1+(\maxpatchsize-1)n}{N-n+1} \sigma_n
 \right)
\end{equation*} 
with $\maxpatchsize$ from \eqref{maxpatchsize}. Consequently, the output trees of Algorithm~\ref{A:conf-tree-algo} are near best with respect to $|\cdot|$ from \eqref{complexity-measure}.
\end{theorem}

\begin{remark}[Abstract nonconforming bisection]
Consider the special case of nonconforming abstract bisection as described in Example \ref{E:nonconf-abstract-bisection}. In view of \eqref{setting-nonconf-abstract-bisection} and \eqref{iteration=complexity}, we have
\begin{equation*}
 \frac{N+1+(\maxpatchsize-1)n}{N-n+1}
 =
 \frac{ \#\leaves(\tree_N) }{ \#\leaves(\tree_N) - n}
\end{equation*}
and, setting $\tilde{N} := \#\leaves(\tree_N)$, the inequality in Theorem~\ref{T:nba-subadditivity} becomes
\begin{equation*}
 \Err(\tree_N)
 \leq
 \min_{m=1}^{\tilde{N}} \left(
 \frac{\tilde{N}}{\tilde{N}-m+1} \sigma_{m-1}
\right)
\end{equation*}
This is exactly Theorem~2.1 in Binev~\cite{binev_tree_2018}, which uses the number of leaves 
as complexity measure.
\end{remark}

We split the proof of Theorem~\ref{T:nba-subadditivity} in several steps. The first one concerns the behavior of the maximal values of the marking indicators arising in line 2 of Algorithm~\ref{A:conf-tree-algo}.
\begin{lemma}[Maximal marking indicators]
\label{L:decreasing-max-ind}
The sequence $(t_n)_{n=0}^N$ of the maximal values of the marking indicators is decreasing.
\end{lemma}

\begin{proof}
We first show a monotonicity property of the marking indicator. To this end, let $\tcell=(\cell,n+1)$ be a tagged cell with $\cell \in \leaves(\tree_{n+1})$ and $n<N^*$ and denote by $\tcell'$ its ``dynamic'' parent, i.e., 
$\tcell'=(\cell',n)$ with either $\cell'=\cell$ or $\cell'=\parent(\cell)$. The subadditivity \eqref{subadditivity} for the local errors and the recursion \eqref{aux-ind;recursion} for the penalization indicators yield
\begin{equation*}
\err(\cell) \leq \err(\cell')
\quad\text{and}\quad
\auxind(\tcell) \geq \auxind(\tcell').
\end{equation*}
The definition \eqref{marking-with-ell} of the marking indicator thus implies
\begin{equation*}
\ind(\tcell)
=
\left(
\frac{1}{\err(\cell)}
+
\auxind(\tcell)
\right)^{-1}
\leq
\left(
\frac{1}{\err(\cell')}
+
\auxind(\tcell')
\right)^{-1}
=
\ind(\tcell').
\end{equation*}
Hence, choosing $\cell = \mcell_{n+1}$, we see that the sequence $(t_n)_n$ is decreasing:
\begin{equation*}
t_{n+1}
=
\ind(\tcell)
\leq
\ind(\tcell')
\leq
t_n.
\qedhere
\end{equation*}
\end{proof}

Our next step summarizes the operations of Algorithm~\ref{A:conf-tree-algo} by defining suitable quantities on the nodes of the tree $\tree_N$. Our starting point is that the set of internal nodes of $\tree_N$ is exactly the union of the selected subdivision patches:
\begin{equation}
\label{all-Internals-dynamic}
 \tree_N \setminus \leaves(\tree_N)
 =
 \bigcup_{\sdpatch\in\npatches} \sdpatch
\quad\text{with}\quad
 \npatches
 :=
 \{\sdpatch_0,\dots, \sdpatch_{N-1}\}.
\end{equation}
Note that 
the subdivision patches $\sdpatch$, $\sdpatch\in\npatches$, are pairwise disjoint. We can group the subdivision patches with respect to the marked cells by introducing the sets
\begin{align*}
 \npatches(\cell)
 :=
 \big\{ \sdpatch_n \mid
  \mcell_n = \cell
  \text{ for some } n \in \{0,\dots,N-1\}
 \big\} 
\end{align*}
of all patches that were subdivided in connection with a given marking cell $\cell\in\tree_N$. Line~4 of Algorithm~\ref{A:conf-tree-algo} ensures that all patches in $\npatches(\cell)$ are necessary for subdividing $\cell$. In addition, we clearly have
\begin{equation*}
 \npatches(\cell) \neq \emptyset
 \iff
 \cell \text{ was marked}
\end{equation*}
and, if nonempty, the subdivision patch $\sdpatch(\cell)$ may or may not be an element of $\npatches(\cell)$.

The fact that the penalization $\auxind$ is only modified for marked cells suggests introducing the \emph{reduced tree} 
\begin{equation}
\label{reduced-tree}
\mtree_N
:=
\big\{ \mcell_0, \mcell_1, \dots, \mcell_{N-1} \big\}
\cup
\leaves(\tree_N),
\end{equation}
where
\begin{equation*}
\text{$\cell'$ is parent of $\cell$ in $\mtree_N$}
\;:\Longleftrightarrow\;
\cell' = \mcmA(\cell)
\end{equation*}
and
\begin{equation}
\label{reduced-parent}
\mcmA(\cell)
:=
\mcell_n
\quad\text{with}\quad
n := \max\big\{ k \mid \mcell_k \in \anc(\cell) \big\}
\end{equation}
is the closest marked ancestor of $\cell$.
%
If $\cell' = \mcmA(\cell)$ for $\cell,\cell'\in\mtree_N$, we also say that $\cell$ is a child of $\cell'$ in $\mtree_N$. Note that then $\cell\subset\cell'$ and, in general, $\cell$ is only a descendant of $\cell'$ (in $T_N$). The number of children of a given node in $\mtree_N$ can vary from $1$ to some maximum number that grows with $N$.  Thus, although the tree $\mtree_N$ shares root and leaves with $\tree_N$, its structure is less rigid than the one of $\tree_N$.
Nevertheless, if $\mtree$ is any subtree of $\mtree_N$ rooted at $\cell$, then 
\begin{equation}
\label{subadditivity-mtree}
 \sum_{\cell'\in\leaves(\mtree)} \err(\cell')
 \leq
 \err(\cell).
\end{equation}
This follows from an induction on the minimal general subtree of $\tree_N$ containing $\mtree$.  

Next, we define ``static'' variants of the ``dynamic'' indicators $\ind$ and $\auxind$.
Given a node $\cell \in \tree_N$, we let
\begin{align*}
 n_{\min}(\cell)
 &:=
 \min \big\{
 n \in \{1,\dots,N\} \mid \cell \in \leaves(\tree_n)
 \big\},
\\
 n_{\max}(\cell)
 &:=
 \max \big
 \{ n \in \{1,\dots,N\} \mid \cell \in \leaves(\tree_n)
 \big\},
\end{align*} 
and set
\begin{equation}
\label{mind}
 \mind(\cell)
 =
 \ind\big( \cell, n_{\max}(\cell) \big)
\quad\text{and}\quad
 \mauxind(\cell)
 = 
 \auxind\big( \cell, n_{\max}(\cell) \big).
\end{equation}
This choice entails that we have, on the one hand,
\begin{subequations}
\label{leaves-vs-marked}
\begin{equation}
\label{leaves-vs-marked;leaves}
 \forall \cell \in \leaves(\mtree_N)
 \quad
 \mind(\cell) \leq t_N
\end{equation}
for the leaves of $\mtree_N$ and, on the other hand,
\begin{equation}
\label{leaves-vs-marked;marked}
 \forall \cell \in \mtree_N
 \quad
 \npatches(\cell) \neq \emptyset
 \implies
 \mind(\cell) \geq t_N
\end{equation}
\end{subequations}
for the marked cells. In fact, if $\cell\in\mtree_N\setminus\leaves(\mtree_N)$ is an internal node, then the implication follows readily from Lemma \ref{L:decreasing-max-ind}, else we have to take into account also the fact that the static marking indicator $\mind$ does not change after the last time it was marked.
 
We will need a recursion for the static penalization indicator $\mauxind$. In order to derive it, let $\cell \in \mtree_N$ be any node of the reduced tree and introduce the following notations. We write
\begin{equation*}
 \apatches(\cell)
 :=
 \npatches(\cell) \setminus \sdpatch(\cell)
\end{equation*}
for the set of the patches that were subdivided at the request of $\cell$ in order to prepare its subdivision. Furthermore, construct a decreasing sequence $(n_j)_{j=0}^J$ by requiring
\begin{gather*}
 n_0 = n_{\max}(\cell),
\quad
 n_{j+1}
 =
 \max \{ k \mid k<n_j, \mcell_k = \cell \} \text{ for }j<J
\intertext{and}
 \{ k \mid k<n_J, \mcell_k = \cell \} = \emptyset,
 \text{ i.e.\ } n_J = n_\text{min}(\cell).
\end{gather*}
Finally, we let $\cell'$ denote the child of $\mcmA(\cell)$ containing $\cell$ and observe the identity $n_{\min}(\cell')-1 = n_{\max}\big(\mcmA(\cell)\big)$.
Applying the four cases of \eqref{aux-ind;recursion} suitably, we derive
\begin{equation*}
\begin{aligned}
 \auxind\big( &\cell, n_{\max}(\cell) \big)
 =
  \auxind(\cell,n_1+1)
 =
 \frac{1}{\err(\cell)}
  +
 \auxind\big( \cell,n_1\big)
\\
 &=
 \frac{J}{\err(\cell)}
  +
 \auxind\big( \cell,n_{J}\big)
 =
 \frac{\#\apatches(\cell)}{\err(\cell)}
 +
 \auxind\big( \cell,n_{\min}(\cell)\big)
\\
 &=
 \frac{\#\apatches(\cell)}{\err(\cell)}
 +
 \auxind\big( \cell',n_{\max}\big(\mcmA(\cell)\big)+1\big)
\\
 &=
 \frac{\#\apatches(\cell)}{\err(\cell)}
  +
 \frac{1}{\err\big(\mcmA(\cell)\big)}
  +
 \auxind\big(\mcmA(\cell),n_{\max}\big(\mcmA(\cell)\big)\big).
\end{aligned}
\end{equation*}
The recursion for the static penalization indicator $\mauxind$ is therefore
\begin{equation}
\label{mauxind-recursion}
 \mauxind(\cell)
 =
 \frac{\#\apatches(\cell)}{\err(\cell)}
 +
 \frac{1}{\err\big(\mcmA(\cell)\big)}
 +
 \mauxind\big(\mcmA(\cell)\big),
\end{equation}
while the static marking indicator satisfies
\begin{equation}
\label{mind-mauxind}
 \mind(\cell)
 =
 \left(
  \frac{1}{\err(\cell)} + \mauxind(\cell)
 \right)^{-1},
\end{equation}
an immediate consequence of the definitions \eqref{mind} and \eqref{marking-with-ell}.

The next two lemmas lay the groundwork to exploit the inequalities in \eqref{leaves-vs-marked}. They are generalizations of Binev~\cite[Lemmas~2.3 and 2.4]{binev_tree_2018} to the setup at hand and make use of the following notation. If $\mtree$ is a general subtree of $\mtree_N$, let
\begin{equation*}
 \npatches(\mtree)
 :=
 \bigcup_{\cell \in \mtree} \npatches(\cell)
\end{equation*}
denote the set of all associated patches that are subdivided in $\tree_N$. Given a subset $\leaves\subseteq\leaves(\tree_N)$ of leaves,
\begin{equation*}
 \sdpatch(\leaves)
 :=
 \{ s(\cell) \mid \cell \in \leaves \}
\end{equation*}
stands for the associated patches, which are not subdivided in $\tree_N$. Note that the subdivision patches of different leaves may coincide, cf.\ \eqref{patch-and-members}, but we have
\begin{equation}
\label{leaves<}
\#\leaves
\leq
\maxpatchsize \#\sdpatch(\leaves)
\end{equation}
with $\maxpatchsize$ from \eqref{maxpatchsize}.
\begin{lemma}[Upper bound for error on leaves]
\label{L:ubd}
Let $\leaves\subseteq\leaves(\tree_N)$ be a subset of leaves of $\tree_N$. Then
\begin{equation*}
 \sum_{\cell\in\leaves} \err(\cell)
 \leq
 \big( \maxpatchsize\#\sdpatch(\leaves) + \#\npatches(\mtree) \big) t_N,
\end{equation*}
where $\mtree$ is the minimal general subtree of the reduced tree $\mtree_N$ that contains $\leaves$ and the root $\Omega$.
\end{lemma}

\begin{proof}
By potentially passing to a subset of $\leaves$, we can assume that the local errors $\err(\cell)$, $\cell \in \leaves$,  never vanishes on the leaves.  Then, in view of \eqref{subadditivity} and \eqref{err,ind=0}, we have that $\err$ and $\mind$ never vanish on the whole tree $\mtree$.

Given a leaf $\cell \in \leaves$, we define a finite sequence $(\cell_j)_{j=0}^J$ by
\begin{equation*}
 \cell_0 = \cell,
\quad
 \cell_{j+1}
 =
 \mcmA(\cell_j) \text{ whenever applicable},
\quad
 \cell_J
 =
 \Omega.
\end{equation*}
Making use of identity \eqref{mind-mauxind}, recursion \eqref{mauxind-recursion}, as well as
\begin{equation*}
 \sdpatch(\cell_0) \not\in \npatches(\cell_0),
\quad
 \sdpatch(\cell_j) \in \npatches(\cell_j) \text{ for }j=1,\dots,J-1,
\quad
 \# \npatches(\cell_J)= 1 
\end{equation*}
and \eqref{aux-ind;init}, we derive
\begin{align*}
 \frac{1}{\mind(\cell_0)}
 &=
 \frac{1}{\err(\cell_0)}
 +
 \mauxind(\cell_0)
\\
 &=
 \frac{1}{\err(\cell_0)}
 +
 \frac{\#\apatches(\cell_0)}{\err(\cell_0)}
 +
 \frac{1}{\err(\cell_1)}
 +
 \mauxind(\cell_1)
\\
 &=
 \frac{1 + \#\npatches(\cell_0)}{\err(\cell_0)}
 +
 \frac{\#\npatches(\cell_1)}{\err(\cell_1)}
 +
 \frac{1}{\err(\cell_2)}
 +
 \mauxind(\cell_2)
 =
 \cdots
\\
 &=
 \frac{1 + \#\npatches(\cell_0)}{\err(\cell_0)}
 +
 \sum_{j=1}^{J-1}
  \frac{\#\npatches(\cell_j)}{\err(\cell_j)}
 +
 \frac{1}{\err(\cell_J)}
 +
 \mauxind(\cell_J) 
\\
 &=
 \frac{1 + \#\npatches(\cell_0)}{\err(\cell_0)}
 +
 \sum_{j=1}^{J}
 \frac{\#\npatches(\cell_j)}{\err(\cell_j)}.
\end{align*}
We multiply by $\mind(\cell_0) \err(\cell_0)$ and get
\begin{align*}
 \err(\cell)
 &=
 \err(\cell_0)
 =
 \mind(\cell_0) \left(
  1 + \#\npatches(\cell_0)
  +
  \sum_{j=1}^{J}
   \#\npatches(\cell_j)\frac{\err(\cell_0)}{\err(\cell_j)}
 \right)
\\
 &=
 \mind(\cell) \left(
  1 + \#\npatches(\cell)
  +
  \sum_{\cell'\in\manc(\cell)}
   \#\npatches(\cell')\frac{\err(\cell)}{\err(\cell')}
  \right),
\end{align*} 
where $\manc(\cell) = \{ \cell_1, \dots, \cell_{J} \}$ stands for the ancestors in the reduced tree $\mtree_N$ of the leaf $\cell=\cell_0$.

We proceed by summing over all leaves $\cell\in\leaves$. Thanks to $\mind(\cell) \leq t_N$ from \eqref{leaves-vs-marked;leaves}, we obtain
\begin{equation}
\label{sum-over-leaves-before-fubini}
 \sum_{\cell\in\leaves} \err(\cell)
 \leq
 t_N \left(
  \#\leaves
  +
  \sum_{\cell\in\leaves} \#\npatches(\cell)
  +
  \sum_{\cell\in\leaves} \sum_{\cell'\in\manc(\cell)}
   \#\npatches(\cell') \frac{\err(\cell)}{\err(\cell')}
  \right).
\end{equation}
In order to bound the critical double sum where cells are hit multiple times, we observe that, for any pair $(\cell,\cell')\in \leaves \times \big(\mtree \setminus \leaves(\mtree)\big)$, we have the equivalence
\begin{equation*}
 \cell \in \leaves \text{ and }\cell' \in \manc(\cell)
 \iff
 \cell' \in \mtree \setminus \leaves
 \text{ and }
 \cell \in \leaves_{\cell'},
\end{equation*}
where $\leaves_{\cell'} := \tree_{\cell'} \cap \leaves$ stands for the descendants of $\cell' \in \mtree \setminus \leaves$ that are in $\leaves$.
Hence, we can reorder the terms and use the subadditivity \eqref{subadditivity-mtree} on the reduced tree to establish
\begin{align*}
 \sum_{\cell\in\leaves} \sum_{\cell'\in\manc(\cell)}
 \#\npatches(\cell') \frac{\err(\cell)}{\err(\cell')}
 =
 \sum_{\cell'\in \mtree \setminus \leaves}
  \#\npatches(\cell')
  \frac{\sum_{\cell\in\leaves_{\cell'}} \err(\cell)}
   {\err(\cell')} 
 \leq
 \sum_{\cell'\in \mtree \setminus \leaves}
  \#\npatches(\cell').
\end{align*}
Inserting the last inequality and \eqref{leaves<} into \eqref{sum-over-leaves-before-fubini}, we arrive at
\begin{equation*}
 \sum_{\cell\in\leaves} \err(\cell)
 \leq
 \left(
  \maxpatchsize \#\sdpatch(\leaves)
  +
  \sum_{\cell\in \mtree}
   \#\npatches(\cell)
  \right) t_N.
\end{equation*}
As the sets $\npatches(\cell)$, $\cell\in\mtree$, are pairwise disjoint,
the proof is finished.
\end{proof}
\begin{lemma}[Lower bound for error of marked cells]
\label{L:lbd}
Let $\cell\in\mtree_N\setminus\{\Omega\}$ be a cell that was marked by Algorithm~\ref{A:conf-tree-algo}. Then
\begin{equation*}
 \err(\cell)
 \geq
 \#\npatches(\mtree)\, t_N,
\end{equation*}
where $\mtree$ is the subtree that consists of $\cell$ and all its descendants in $\mtree_N$.
\end{lemma}

\begin{proof}
We assume $t_N>0$, excluding the trivial case $t_N=0$. Consequently, $\err$ and $\mind$ never vanish on $\mtree \setminus \leaves(\mtree)$.

We shall employ a tree induction to prove
\begin{equation}
\label{prove-by-tree-induction}
 \err(\cell)
 \geq
 t_N \sum_{\cell'\in\mtree} \#\npatches(\cell').
\end{equation}
The claim then follows because the sets $\npatches(\cell)$, $\cell\in\mtree$, are pairwise disjoint. In order to verify \eqref{prove-by-tree-induction}, we start by considering $\mtree$ which consists of a leaf $\cell$ of $\mtree_N$. We establish a slightly stronger variant of \eqref{prove-by-tree-induction}, incorporating also the case of leaves that were never marked.

\emph{Case 1:} $\cell$ was marked at least once. Let $n$ be the last iteration when $\cell$ was marked, i.e.\ $\ind(\cell,n)\geq t_N$. Since the marking indicator changes only if $\cell$ is marked or subdivided, we deduce $\mind(\cell) \geq t_N$. Therefore, identity \eqref{mind-mauxind} for the static marking indicator, $\cell \neq \Omega$, recursion \eqref{mauxind-recursion} for static penalization indicator, and $\apatches(\cell) = \npatches(\cell)$ give
\begin{subequations}
\label{base-case-full-bound}
\begin{equation}
\label{base-case-full-bound;marked}
\begin{aligned}
 \err(\cell)
 =
 \mind(\cell) \frac{\err(\cell)}{\mind(\cell)}
 &\geq
 t_N \left(
  1 
  +
  \#\apatches(\cell)
  +
  \frac{\err(\cell)}{\err\big(\mcmA(\cell)\big)}
  +
  \err(\cell) \mauxind\big(\mcmA(\cell)\big)
 \right)
\\
 &\geq
 t_N \left(
  \#\npatches(\cell)
  +
  \frac{\err(\cell)}{\err\big(\mcmA(\cell)\big)} 
  +
  \err(\cell)\mauxind\big(\mcmA(\cell)\big)
 \right).
\end{aligned}
\end{equation}

\emph{Case 2:} $\cell$ was never marked. As $\cell \neq \Omega$, \eqref{leaves-vs-marked;marked} and $\#\npatches(\cell)=0$ allow us to write
\begin{equation}
\label{base-case-full-bound;unmarked}
 \err(\cell)
  =
 \mind\big( \mcmA(\cell) \big)
  \frac{\err(\cell)}{\mind\big( \mcmA(\cell) \big)}
 \geq
 t_N \left(
  \#\npatches(\cell)
  +
  \frac{\err(\cell)}{\err\big(\mcmA(\cell)\big)}
  +
  \err(\cell) \mauxind\big(\mcmA(\cell)\big)
 \right).
\end{equation}
\end{subequations}

The two inequalities \eqref{base-case-full-bound;marked} and \eqref{base-case-full-bound;unmarked} together constitute the base case of our tree induction.

We turn to the induction step. In accordance with \eqref{base-case-full-bound}, we assume that we are given $J\in\N$ and finite sequences $(\cell_j)_{j=1}^J$ and $(M_j)_{j=1}^J$ such $\mcmA(\cell_j)=\cell$ and
\begin{equation}
\label{tree-induction-hypothesis}
 \err(\cell_j)
 \geq
 t_N \left(
  M_j
  +
  \frac{\err(\cell_j)}{\err(\cell)}
  +
  \err(\cell_j) \mauxind(\cell)
\right).
\end{equation}
We add these inequalities and multiply by $\err(\cell) / \big( \sum_{j=1}^J \err(\cell_j) \big)$ to obtain
\begin{equation*}
 \err(\cell)
 \geq
 t_N \left(
 \frac{\err(\cell)}{\sum_{j=1}^J \err(\cell_j)}
  \sum_{j=1}^J
   M_j + 1 + \err(\cell) \mauxind(\cell)
\right).
\end{equation*}
Hence, subadditivity \eqref{subadditivity-mtree} on the reduced tree implies
\begin{equation*}
 \err(\cell)
 \geq
 t_N \left(
  \sum_{j=1}^J
   M_j + 1 + \err(\cell) \mauxind(\cell)
 \right).
\end{equation*}
Using $\cell\neq\Omega$ and the recursion \eqref{mauxind-recursion} for $\mauxind$ another time and $1 + \#\apatches(\cell) = \#\npatches(\cell)$, we arrive at
\begin{equation*}
 \err(\cell)
 \geq
 t_N \left(
  M
  +
  \frac{\err(\cell)}{\err\big(\mcmA(\cell)\big)}
  +
  \err(\cell) \mauxind\big(\mcmA(\cell)\big)
 \right).
\end{equation*}
with
\begin{equation*}
 M
 =
 \#\npatches(\cell) + \sum_{j=1}^J M_j. 
\end{equation*}
The last inequality is the counterpart of the induction hypothesis \eqref{tree-induction-hypothesis} and implies \eqref{prove-by-tree-induction} after making $M$ explicit.
\end{proof}

After these preparations we are ready for the proper proof of Theorem~\ref{T:nba-subadditivity}.

\begin{proof}[Proof of Theorem \ref{T:nba-subadditivity}]
To bound $\Err(\tree_N)$ in terms of $\sigma_n$ for a given $n\in\{0,\dots, N\}$, we compare the generated tree $\tree_N$ with a \emph{conforming} best approximation tree $\tree_n^\star$ satisfying
\begin{equation*}
 \Err(\tree_n^\star) = \sigma_n
\quad\text{and}\quad
 |\tree^\star_n| = \#\npatches^\star \leq n,
\end{equation*}
where $\npatches^\star$ denotes the set of the patches that have been subdivided in the creation of $\tree_n^\star$. Therefore,
\begin{equation*}
 \tree_n^\star \setminus \leaves(\tree_n^\star)
 =
 \bigcup_{\sdpatch \in \npatches^\star} \sdpatch.
\end{equation*}

Let us start by observing that the error of certain leaves $\cell\in\leaves(\tree_N)$ is readily bounded with the help of subadditivity. Indeed, if $\cell \not\in \tree_n^\star \setminus \leaves(\tree_n^\star)$, then $\cell$ is a leaf of a subtree $\tree_{\cell^\star} \cap \tree_N$, which consists of a leaf $\cell^* \in \leaves(\tree_n^\star)$ of the best approximation tree and all its descendants in $\tree_N$. Applying subadditivity \eqref{subadditivity} to all such subtrees, we deduce
\begin{equation}
\label{easy-bound}
 \sum_{\cell\in\leaves_0} \err(\cell)
 \leq
 \sigma_n 
\end{equation}
with
\begin{align*}
 \leaves_0
 :=
 \big\{ \cell \in \leaves(\tree_N) \mid
  \cell \not\in \tree_n^\star \setminus \leaves(\tree_n^\star)
 \big\}.
\end{align*}
Moreover, if $\npatches^\star \subseteq \npatches$, then $\tree_n^\star \setminus \leaves(\tree_n^\star) \subseteq \tree_N \setminus \leaves(\tree_N)$ and \eqref{easy-bound} becomes $\Err(\tree_N) \leq \sigma_n$, which implies the desired bound. We are therefore left with the case
\begin{equation}
\label{one-subdivision-missed}
 \npatches^\star \setminus \npatches \not= \emptyset,
\end{equation}
which entails in particular $n\geq2$.

In order to proceed, we first derive a lower bound for the best error $\sigma_n$ in terms of $N$, $n$, and the maximal indicator value $t_N=\max\{ \mind(\cell) \mid \cell\in\tree_N\}$ on the leaves of the tree $\tree_N$. In view of \eqref{one-subdivision-missed}, $\npatches$ and $\npatches^\star$ have at most $n-1$ subdivision patches in common and so
\begin{equation*}
 \#(\npatches \setminus \npatches^\star) \geq N-n+1.
\end{equation*}
For every patch $\sdpatch \in \npatches\setminus\npatches^\star$, there exists a marked cell $\mcell_\sdpatch \in \mtree_N$ with $\sdpatch \in \npatches(\mcell_\sdpatch)$. As $\sdpatch \not\in \npatches^\star$, we have $\sdpatch \cap \big( \tree_n^\star \setminus \leaves(\tree_n^\star) \big) = \emptyset$ because of \eqref{patch-and-members} and therefore $\mcell_\sdpatch \not\in \tree_n^\star \setminus \leaves(\tree_n^\star)$ by \eqref{locating-subdivision-patches}, the fact that $\sdpatch$ is necessary for subdividing for $\cell^\star_\sdpatch$, and the conformity of $\tree_n^\star$. Hence, 
\begin{equation*}
 \npatches \setminus \npatches^\star
 \subseteq
 \bigcup_{\cell \in \mtree_N \setminus (\tree_n^\star \setminus \leaves(\tree_n^\star) )} \npatches(\cell).
\end{equation*}
The set $\mtree_N \setminus \big( \tree_n^\star \setminus \leaves(\tree_n^\star) \big)$ is a union of subtrees in $\mtree_N$. We collect in the set $R$ the roots of these subtrees.  For each $\cell \in R$, there exists a unique leaf $\cell^\star \in \leaves(\tree_n^\star)$ of the best approximation tree such that $\cell^\star \in \{\cell\} \cup \mathcal{A}(\cell)$. Hence, employing the subadditivity \eqref{subadditivity} of the local errors and applying the local lower bound in Lemma \ref{L:lbd} on each cell root in $R$, we obtain
\begin{equation}
\label{glb}
\begin{aligned}
 \sigma_n
 &\geq
 \sum_{\cell\in R} \err(\cell)
 \geq
 \sum_{\cell \in R} \# \npatches(\tree_\cell \cap \mtree_N) t_N
 =
 \#\npatches\Big( \mtree_N \setminus \big( \tree_n^\star \setminus \leaves(\tree_n^\star) \big) \Big) t_N
\\
 &\geq
 \# (\npatches \setminus \npatches^\star) t_N
 \geq
 (N-n+1) t_N.
\end{aligned}
\end{equation}

Next, we complement the partial upper bound \eqref{easy-bound}. To this end, we first employ the upper bound in Lemma~\ref{L:ubd} with $\leaves := \leaves(\tree_N) \setminus \leaves_0$. Observe that, in the notation of Lemma~\ref{L:ubd}, the sets $\sdpatch(\leaves)$ and $\npatches(\mtree)$ are disjoint and, thanks to $\leaves \subset \tree^\star_n \setminus \leaves(\tree^\star_n)$ and \eqref{locating-subdivision-patches}, their union satisfies $\sdpatch(\leaves) \cup \npatches(\mtree) \subseteq \tree_n^\star \setminus \leaves(\tree_n^\star)$. Recalling \eqref{patch-and-members}, we therefore find
\begin{equation*}
 \sum_{\cell\in\leaves} \err(\cell)
 \leq
 \maxpatchsize \, \#\big( \sdpatch(\leaves) \cup \npatches(\mtree) \big) \, t_N
 \leq
 \maxpatchsize \, \#\npatches^\star \, t_N
 \leq
 \maxpatchsize n t_N.
\end{equation*}
Inserting the lower bound \eqref{glb}, we arrive at
\begin{equation}
\label{2nd-ubd}
 \sum_{\cell\in\leaves} \err(\cell)
 \leq
 \frac{\maxpatchsize n}{N-n+1} \sigma_n.
\end{equation}
Thus, combining the two upper bounds \eqref{easy-bound} and \eqref{2nd-ubd}, we conclude
\begin{equation*}
 \Err(T_N)
 \leq
 \sigma_n \left(
  1 + \frac{\maxpatchsize n}{N-n+1}
 \right)
 =
 \frac{N+1 + (\maxpatchsize-1) n}{N-n+1} \sigma_n.
\qedhere
\end{equation*}
\end{proof}

%
%
%
%

%
%
%
%

%
%
%
\section{FEM-motivated numerical examples}
\label{S:Tests}
%
%
This section complements our theoretical results with several numerical examples. The examples focus on the adaption of the marking indicators to conforming meshes. More precisely, they contrast Algorithm~\ref{A:conf-tree-algo} with Algorithm~\ref{A:tree-algo}, whose marking indicators do not take the conformity of the meshes into account.
\begin{algorithm}
	\caption{Tree algorithm for conforming meshes with simple marking}	
	\label{A:tree-algo}
	\begin{algorithmic}[1]
		\Statex
		\Require master tree $\mastertree$ with local error functional $\err$
		\Statex
		\Statex \Comment{initialization}
		
		\State $\tree_0 \gets \{\Omega\}$, $\ind(\Omega) \gets \err(\Omega)$, $n \gets 0$
		\Statex
		\While {$t_n := \max_{\cell\in\leaves(\tree_n)} \ind(\cell)>0$}
		\Comment{or some alternative stopping test}
		\State mark some $\mcell_n \in \leaves(\tree_n)$ with $\ind(\mcell_n) = t_n$
		\State pick a subdivision patch $\sdpatch_n \subseteq \leaves(\tree_n)$ that is necessary for subdividing $\mcell_n$
		\State define $\tree_{n+1}$ by adding all children $C_n$ of the cells in $\sdpatch_n$ to $\tree_n$
		\Statex
		\Statex \Comment{penalize the marking indicators of the subdivided cells}
		\ForAll {$\cell \in C_n$}
		\State $\ind(\cell) \gets \left( \displaystyle
		\frac{1}{\err(\cell)} + \frac{1}{\ind\big(\parent(\cell)\big)}
		\right)^{-1}$
		\EndFor
		\State $n \gets n+1$
		\EndWhile
	\end{algorithmic}
\end{algorithm}

The comparison is conducted 
with newest-vertex bisection for edge-to-edge triangulations; cf.\ Example \ref{E:nvb-triangulations}. 
We shall briefly indicate the interest of the setup 
in the numerical solution of partial differential equations. 

Before turning to the numerical examples, let us observe that the abstract setup of \S\ref{S:TreeApprox-for-ConfMeshes} covers also initial meshes with several cells. In applications, this fact is in particular useful when the domain $\Omega$ has a more complicated shape. 
\begin{remark}[Arbitrary initial meshes]
\label{R:arb-init-mesh}
Assume that $\imesh$ is any finite partition of a domain $\Omega \subseteq \Rd$ and that $\tree_{\cell}$ is an infinite tree without leaves for each cell $\cell \in \imesh$. We then introduce a master tree $\mastertree$ by taking $\Omega$ as root, the elements in $\imesh$ as children of $\Omega$, and, for each $\cell\in\imesh$, the descendants in $\tree_{\cell}$ as descendants of $\cell$. Since there is a one-to-one correspondence between subforests of $\bigcup_{\cell\in\imesh} \tree_{\cell}$ and subtrees of $\mastertree$, notions such as ``conforming'' given in terms of the former naturally carry over to the latter, without changing $\maxpatchsize$ from \eqref{maxpatchsize}. Furthermore, if $\err$ is an error functional on  $\bigcup_{\cell\in\imesh} \tree_{\cell}$, we extend it on $\mastertree$ by $\err(\Omega) := \infty$. Then the first iteration of Algorithm~\ref{A:conf-tree-algo} generates $\imesh$ with the assignments $\auxind(\cell) = 0$ and $\ind(\cell) = \err(\cell)$ for all $\cell \in \imesh$, which appears to be the ``natural'' starting point for multi-cell initial meshes. 
\end{remark}

The numerical results have been obtained with implementations in the framework of the finite element toolbox \textsf{ALBERTA} \cite{schmidt_design_2005}. Local errors are approximated using straight-forward numerical integration of order 17.

\subsection{$H^1$-approximation with Crouzeix-Raviart or Lagrange elements}
\label{S:H1Approx}
%
%

Suppose that $\imesh$ is an edge-to-edge triangulation of the planar domain $\Omega\subset\R^2$, where $\partial\Omega$ is locally a Lipschitz graph; the latter assumption serves to avoid technicalities at the boundary; cf.\ Stevenson \cite[\S3]{stevenson_completion_2008} or Veeser and Zanotti \cite{veeser_quasi-optimal_boundary_2019} for boundaries that are not locally a graph. Owing to Binev et al.\ \cite[Lemma~2.1]{binev_adaptive_2004}, we can choose orderings of the vertices in each triangle of $\imesh$ such that if the intersection of two triangles $\cell_1,\cell_2\in\imesh$ is an edge, then it is either the refinement edge of both or none:
\begin{equation}
\label{admis-init-triangulation}
 \forall \cell_1,\cell_2 \in \imesh
\quad
 \redge(\cell_1) = \cell_1 \cap \cell_2
 \iff
 \redge(\cell_2) = \cell_1 \cap \cell_2.
\end{equation}
The vertex orderings and newest vertex bisection induce in particular infinite binary trees $\tree_{\cell}$, $\cell \in \imesh$, and we can
define $\mastertree$ as in Remark \ref{R:arb-init-mesh} and $\ctrees$ as in Example~\ref{E:nvb-triangulations}. Exploiting \eqref{admis-init-triangulation}, assumptions \eqref{sdpatch-for-cell} and \eqref{existence-necessary-subdivison-patches} again follow from Binev at al.\ \cite[Lemma~2.2~(ii) and (iii)]{binev_adaptive_2004}.

Theorem 2.4 of Binev et al.\ \cite{binev_adaptive_2004} provides the following important fact about edge-to-edge triangulation induced by newest vertex bisection. There is a constant $C_\text{cmpl} \geq 1$ such that if $\tree\subset\mastertree$ is any finite subtree rooted at $\Omega$, the smallest conforming tree $T' \in \ctrees$ containing $\tree$ satisfies
\begin{equation}
\label{completion-constant}
 \#\tree' \leq C_\text{cmpl} \#\tree. 
\end{equation}
The constant $C_\text{cmpl}$, called \emph{completion constant}, plays a key role in our comparison. Considering concrete examples, one obtains $C_\text{cmpl}\geq14$.

In order to define the local errors, we fix some function $u \in H^1_0(\Omega)$ with square-integrable weak derivatives of first order and vanishing trace on $\partial\Omega$. The squared best errors 
\begin{equation}
\label{H1-loc-err}
 \err(\cell)
 :=
 \inf_{p \in \Poly_1(\cell)} \norm{\nabla(u-p)}{L^2(\cell)}^2,
\quad
 \cell \in \mastertree,
\end{equation}
with polynomials of degree $\leq1$ in the $H^1(\cell)$-seminorm are then the local errors. Clearly, they are subadditive.

The interest in this setup originates in the global errors $\Err(\mesh) = \sum_{\cell\in\mesh} \err(\cell)$, where $\mesh = \leaves(\tree)$, $\tree \in \ctrees$, is an edge-to-edge triangulation.
To illustrate this, we write
\begin{equation*}
 \Poly_1(\mesh)
 :=
 \{ v \in L^\infty(\Omega) \mid
 \forall \cell \in \mesh\; v_{|\cell} \in \Poly_1(\cell) \}
\end{equation*}
and consider the following finite-dimensional spaces over any edge-to-edge triangulation $\mesh$:
\begin{align*}
 V_1(\mesh) 
 &:=
 \{ v \in \Poly_1(\mesh) \mid
  v \in C^0(\bar{\Omega}),\, v_{|\partial\Omega} = 0\},
\\
 V_2(\mesh)
 &:=
 \{ v \in \Poly_1(\mesh) \mid
  \text{$v$ is continuous in $\mathcal{C}_\mesh \cap \Omega$},\,
  v_{|\mathcal{C}_\mesh \cap \partial\Omega} = 0 \},
\end{align*}
where $\mathcal{C}_\mesh$ is the set of all edge midpoints of $\mesh$. Both spaces are standard finite element spaces. They can be used, for example, to  solve approximately the Poisson problem or other linear elliptic boundary value problem of second order. For the \emph{continuous linear finite element space} $V_1(\mesh) \subseteq H^1_0(\Omega)$, the well-known C\'ea lemma then implies that the error in the $H^1_0$-seminorm of the approximate solution is near best in the sense that it is bounded by the best error in $V_1 $; cf.\ Brenner and Scott~\cite[(2.8.1)]{brenner_mathematical_2008}. For the \emph{Crouzeix-Raviart space} $V_2(\mesh) \not\subseteq H^1_0(\Omega)$, this is more delicate and has been achieved only recently by Veeser and Zanotti \cite[Theorem~3.4]{veeser_quasi-optimal_over_2019}.
The best errors in these spaces in turn satisfy
\begin{equation}
\label{best-errors-linear-fe}
 \Err(\mesh)
 =
 \inf_{v \in V_2(\mesh)} \norm{\nabla_\mesh(u-v)}{L^2(\Omega)}^2
 \simeq
 \inf_{v \in V_1(\mesh)} \norm{\nabla(u-v)}{L^2(\Omega)}^2,
\end{equation}
where $\nabla_\mesh$ is the broken gradient given by $(\nabla_\mesh v)_{|\cell} = \nabla (v_{|\cell})$ for all $\cell\in\mesh$.
Here the identity for $V_2(\mesh)$ readily follows from the well-known fact that the Crouzeix-Raviart interpolant is a best approximation, see, e.g., \cite[Lemma~3.2]{veeser_quasi-optimal_over_2019}, while the equivalence for $V_1(\mesh)$ is established in Veeser \cite[Corollary~1]{veeser_approximating_2016}, with a multiplicative constant in the upper bound depending on the shape regularity of mesh $\mesh$.

In the light of \eqref{best-errors-linear-fe}, conforming tree approximation with the given setting computes approximations of the nonlinear best errors
\begin{equation*}
 \inf_{\tree \in \ctrees_n} \inf_{v \in V_i(\tree)} \norm{\nabla_T(u-v)}{L^2(\Omega)}^2,
\quad
 n \in \N_0,
\quad
 i=1,2,
\end{equation*}
where we write $\tree$ in place of $\leaves(\tree)$ for short. This creates benchmarks for testing algorithms that adaptively solve the Poisson problem 
whenever the exact solution is known.

Another, more basic, application of Algorithm~\ref{A:conf-tree-algo} is \emph{coarsening} or \emph{sparsifying}. This can be used within the adaptive solution of  boundary value problems by iterating the main steps
\begin{equation*}
\text{errore reduction $\to$ sparsity adjustment.}
\end{equation*}
For the Poisson problem, Binev et al.~\cite{binev_adaptive_2004} propose and analyze such an algorithm, employing the standard finite element spaces $V_1(\tree)$ and an algorithm closely related to Algorithm~\ref{A:tree-algo}.

For both applications, the near-best property is instrumental and holds for Algorithm~\ref{A:conf-tree-algo} and \ref{A:tree-algo}. More precisely, if $\tree^1_N$ is an output tree of Algorithm~\ref{A:conf-tree-algo}, then Theorem~\ref{T:nba-subadditivity} implies in particular
\begin{subequations}
\label{H1:A1,A2}
\begin{equation}
\label{H1:A1}
 \Err(\tree^1_N)
 \leq
 4 \inf \big\{ \Err(\tree) \mid
  \tree \in \ctrees,\; |\tree| \leq \tfrac{1}{2} N 
 \big\}.
\end{equation}
In the case of Algorithm~\ref{A:tree-algo}, we have
\begin{equation}
\label{H1:A2}
\Err(\tree^2_N)
\leq
4 \inf \left\{ \Err(\tree) \mid
 \tree \in \mathbb{T},\; \#\tree \leq \tfrac{3}{4} C_\text{cmpl}^{-1} N 
\right\},
\end{equation}
\end{subequations}
where $\mathbb{T}$ stands for the set of all subtrees of $\mastertree$ that are finite, full, and rooted at $\Omega$. This follows from \cite[Theorem~1.2]{binev_tree_2018} and \eqref{completion-constant}; cf.\ \cite[Theorem~7]{veeser_approximating_2016} and note that the algorithm therein is a reformulation of Algorithm~\ref{A:tree-algo} for subadditive local errors.

Comparing the two results \eqref{H1:A1,A2}, we observe differences in the complexity measure, the competing trees, and the dependences of the constant $C_2$ in the definition \eqref{def-near-best} of ``near best''. The difference in the complexity measure appears not to be so important as $\#\tree \leq |\tree| \leq 2 \#\tree$ for any $\tree \in \ctrees$. Note that the other two differences are interrelated. In fact, the absence of the edge-to-edge constraint on the right-hand side of \eqref{H1:A2} justifies the presence of the completion constant $C_\text{cmpl}\geq14$. These differences appear to be more important, in particular for other applications where $C_\text{cmpl}$ can be quite large; cf.\ Demlow and Stevenson~\cite{demlow_convergence_2011}.

Let us now see how these theoretical differences play out in practice by approximating two concrete functions.
The first one is
\begin{equation}
\begin{aligned}
 u_1(x) &:=  r^{2/3} \sin\left(\tfrac{2}{3} \theta\right),
\\
 &x = r(\cos\theta,\sin\theta) \in \Omega_1 := \opnint{{-\tfrac{1}{2}}}{\tfrac{1}{2}}^2 \setminus \left(\clsint{0}{\tfrac{1}{2}}\times\clsint{{-\tfrac{1}{2}}}{0}\right)
\end{aligned}
\end{equation}
and models a point singularity of the solution of a partial differential equation, which can be triggered by a reentrant corner in the domain boundary. Figure~\ref{F:u1} shows the graph of $u_1$ and compares the convergence histories of Algorithm~\ref{A:conf-tree-algo} and \ref{A:tree-algo}. 
\begin{figure}
\centering
\hfill
\includegraphics[width=0.4\hsize]{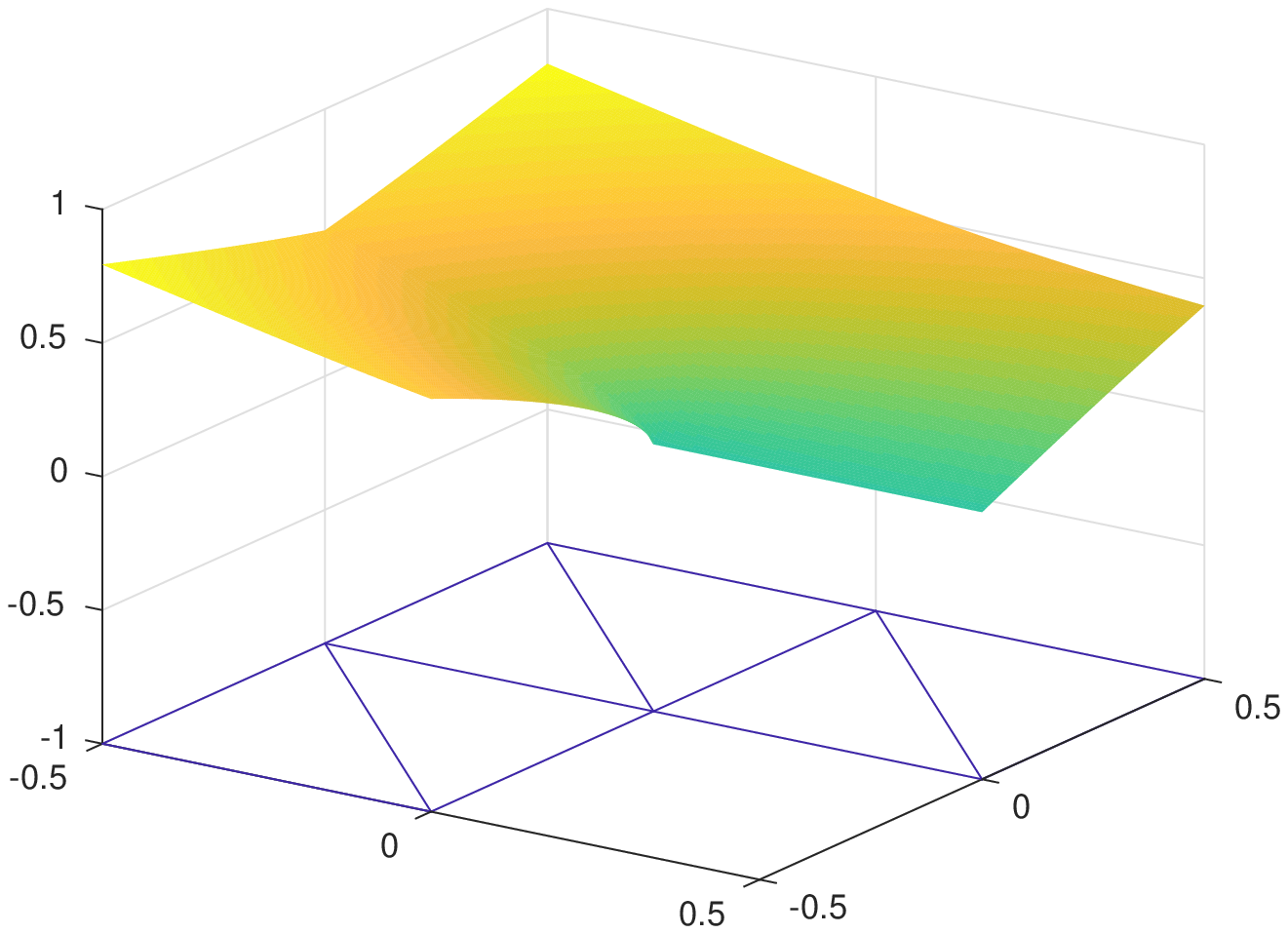}
\hfill
\includegraphics[width=0.4\hsize]{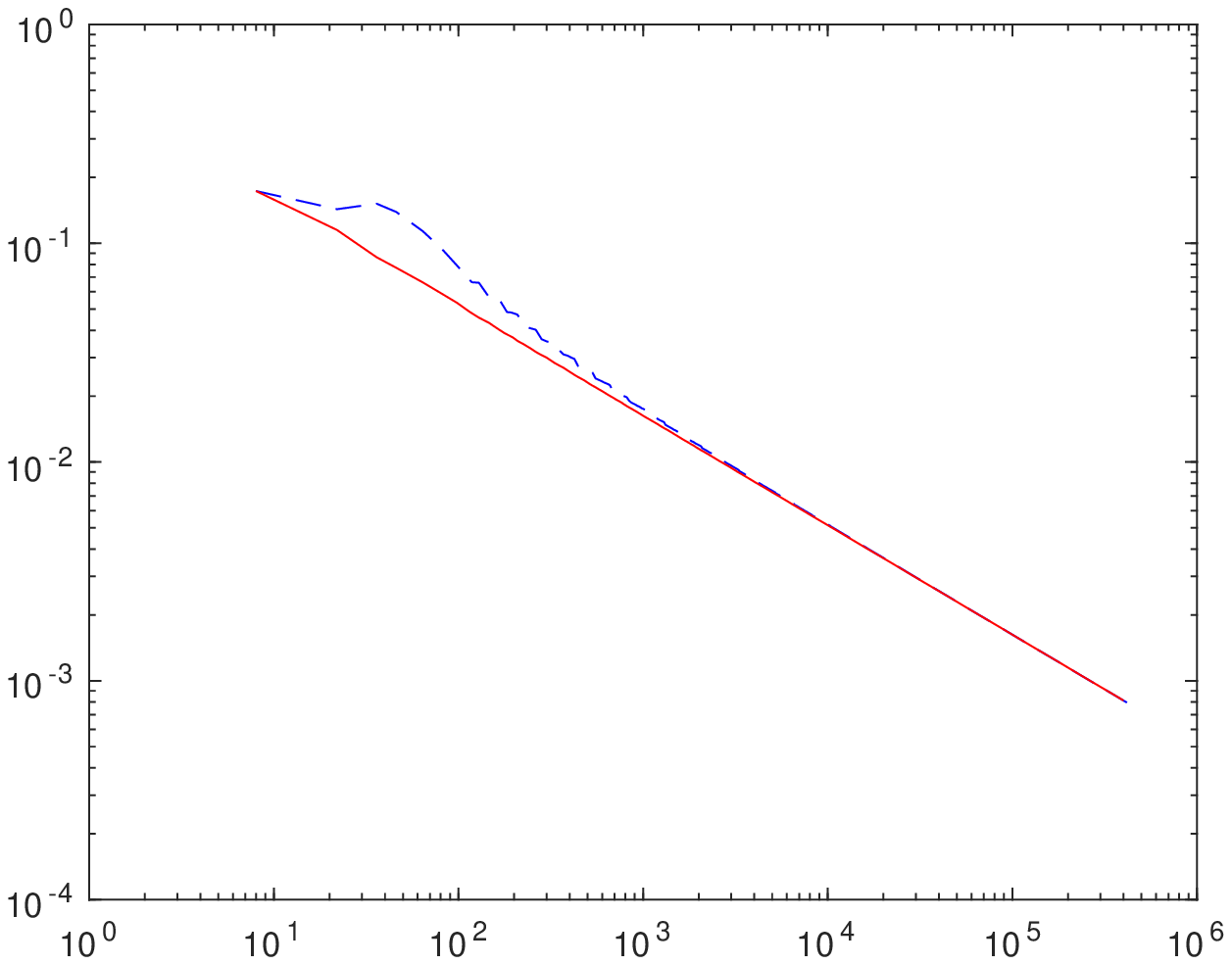}
\hfill
\caption{Approximating $u_1$: (left) graph of $u_1$ with initial triangulation and (right) global $H^1$-error versus cardinality of triangulation for Algorithms \ref{A:conf-tree-algo} (continuous red line) and \ref{A:tree-algo} (dashed blue line) in log-log scale.}
\label{F:u1}
\end{figure}
The lines interpolate the global errors of every 10th iteration. We observe that Algorithm~\ref{A:conf-tree-algo} immediately picks up the asymptotic speed
$-1/2$ and thus has slightly better approximation properties for triangulations with small cardinality. However, for triangulations with large cardinality the approximation properties of both algorithms essentially coincide.

The second target function is
\begin{equation}
 u_2(x)
 := 
 \max \left\{ 0, \tfrac{1}{9} - |x|^2 \right\}, 
\quad
 x = (x_1,x_2) \in \Omega_2 := \opnint{{-1}}{1}^2,
\end{equation}
which is quadratic in the ball $\omega_2 := \{ x\in \R^2 \mid |x| < \tfrac{1}{3} \}$ and vanishes elsewhere. This function is a ``partially trivial target''. In fact, $u_2$ could be approximated by subdividing only cells intersecting $\omega_2$, but the requirement of edge-to-edge conformity entails that refinement spreads out a little to the complement of $\omega_2$. Functions of this type arise, e.g., as solutions of suitable elliptic obstacle problems. Figure \ref{F:u2} shows for this example again graph, initial triangulation, and convergence histories. Here we see that Algorithm~\ref{A:conf-tree-algo} has better approximation properties than Algorithm~\ref{A:tree-algo} persistently.

%
\begin{figure}
\centering
\hfill
\includegraphics[width=0.4\hsize]{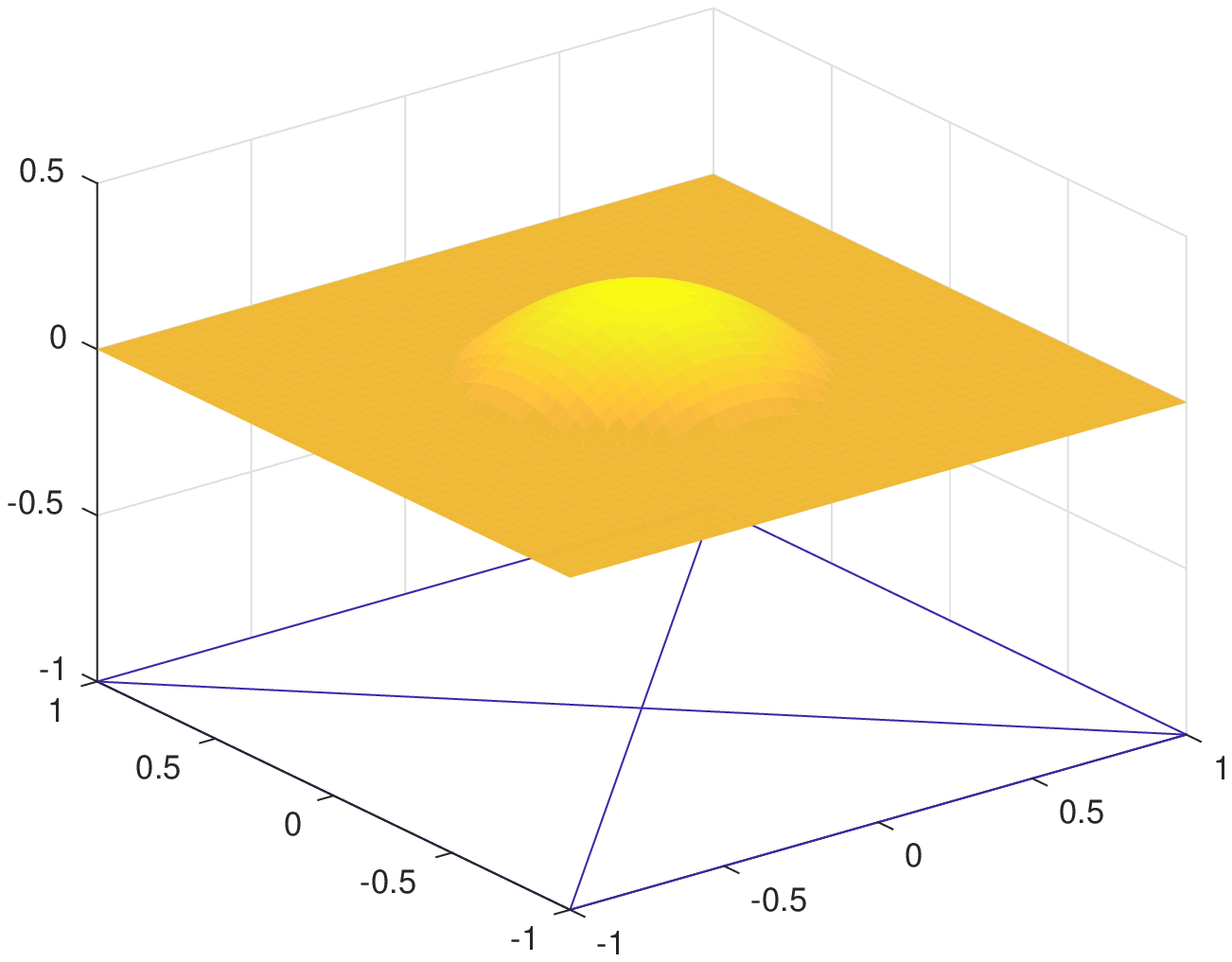}
\hfill
\includegraphics[width=0.4\hsize]{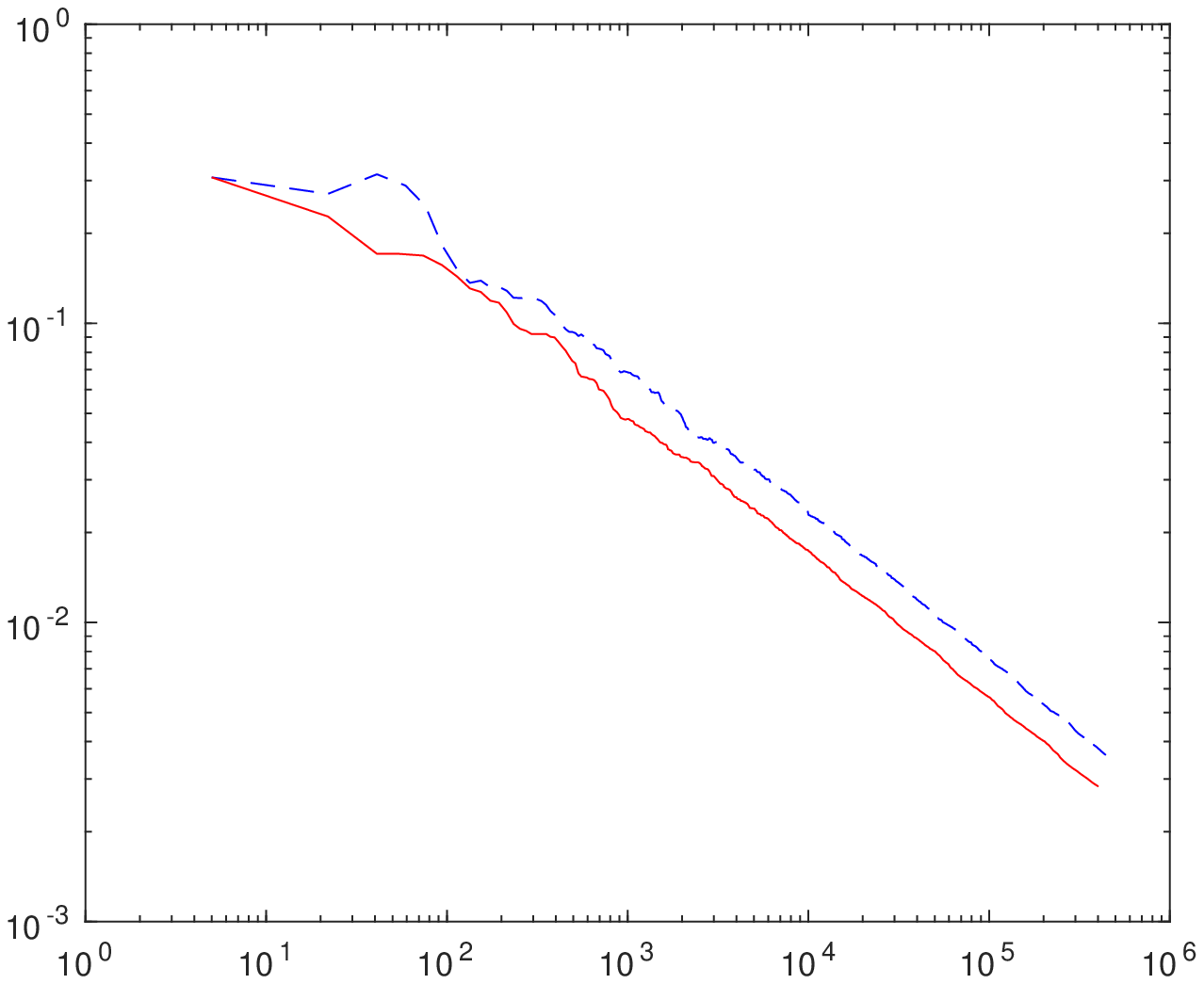}
\hfill
\caption{Approximating $u_2$: (left) graph of $u_2$ with initial triangulation and (right) global $H^1$-error versus cardinality of triangulation for Algorithms \ref{A:conf-tree-algo} (continuous red line) and \ref{A:tree-algo} (dashed blue line) in log-log scale.}
\label{F:u2}
\end{figure}
%

%
%
%
%
%

%
\subsection*{Acknowledgment}
This research has been supported by NSF-DMS~1720297, Italian PRIN 2017 NA-FROM-PDEs, and the Italian GNCS.

We would like to thank Alfred Schmidt for leaving us some source code that was useful for implementing Algorithm \ref{A:conf-tree-algo} within the finite element toolbox \textsf{ALBERTA}.
\bibliographystyle{siam}
\bibliography{ms}

\end{document}